\documentclass[11pt,leqno]{article}
\usepackage[margin=1in]{geometry} 
\usepackage{amssymb,amsfonts,amsmath,bbm,mathrsfs,stmaryrd,mathtools}
\usepackage{xcolor}
\usepackage{url}

\usepackage{graphicx}

\usepackage{accents}

\usepackage{extarrows}

\usepackage[shortlabels]{enumitem}
\usepackage{tensor}

\usepackage{xr}
\usepackage[T1]{fontenc}
\usepackage[utf8]{inputenc}

\usepackage[colorlinks,
linkcolor=black!75!red,
citecolor=blue,
pdftitle={},
pdfproducer={pdfLaTeX},
pdfpagemode=None,
bookmarksopen=true,
bookmarksnumbered=true,
backref=page]{hyperref}

\usepackage{tikz}
\usetikzlibrary{arrows,calc,decorations.pathreplacing,decorations.markings,decorations.shapes,intersections,shapes.geometric,through,fit,shapes.symbols,positioning,decorations.pathmorphing}

\makeatletter
\newlength\zig@L
\newlength\zig@La
\newlength\zig@Lb

\newcommand{\xzigrightarrow}[2][]{%
  \mathrel{%
    \settowidth{\zig@La}{$\scriptstyle #2$}%
    \settowidth{\zig@Lb}{$\scriptstyle #1$}%
    \zig@L=\zig@La\relax
    \ifdim\zig@Lb>\zig@L \zig@L=\zig@Lb\fi
    \advance\zig@L by 2.2em\relax
    \tikz[baseline=-0.65ex]{%
      \draw[->,
            line cap=round,
            decorate,
            decoration={zigzag,segment length=4pt,amplitude=1.1pt}]%
        (0,0) -- (\zig@L,0)
        node[midway,above=2pt] {$\scriptstyle #2$}%
        \if\relax\detokenize{#1}\relax\else
          node[midway,below=2pt] {$\scriptstyle #1$}%
        \fi
      ;
    }%
  }%
}
\makeatother

\makeatletter
\newcommand{\squigjoin}{1mu} 

\def\sqleft@{\sim}                    
\def\sqmid@{\sim\mkern-\squigjoin}    

\def\rightsquigarrowfill@{%
  \arrowfill@{\sqleft@}{\sqmid@}{\mkern-4mu\succ}%
}

\newcommand{\xrightsquigarrow}[2][]{%
  \ext@arrow 0359\rightsquigarrowfill@{#1}{#2}%
}
\makeatother

\makeatletter
\newcommand*\circled[1]{\tikz[baseline=(char.base)]{
    \node[shape=circle, draw, inner sep=0pt, 
    minimum height={\f@size},] (char) {\vphantom{WAH1g}#1};}}
\makeatother

\makeatletter
\DeclareRobustCommand\widecheck[1]{{\mathpalette\@widecheck{#1}}}
\def\@widecheck#1#2{%
    \setbox\z@\hbox{\m@th$#1#2$}%
    \setbox\tw@\hbox{\m@th$#1%
       \widehat{%
          \vrule\@width\z@\@height\ht\z@
          \vrule\@height\z@\@width\wd\z@}$}%
    \dp\tw@-\ht\z@
    \@tempdima\ht\z@ \advance\@tempdima2\ht\tw@ \divide\@tempdima\thr@@
    \setbox\tw@\hbox{%
       \raise\@tempdima\hbox{\scalebox{1}[-1]{\lower\@tempdima\box
\tw@}}}%
    {\ooalign{\box\tw@ \cr \box\z@}}}
\makeatother

\usepackage{braket}

\usepackage[amsmath,thmmarks,hyperref]{ntheorem}
\usepackage{cleveref}

\newcommand\nthalias[1]{\AddToHook{env/#1/begin}{\crefalias{lemma}{#1}}}

\nthalias{definition}
\nthalias{example}
\nthalias{examples}
\nthalias{remark}
\nthalias{remarks}
\nthalias{convention}
\nthalias{notation}
\nthalias{construction}
\nthalias{sketch}
\nthalias{theoremN}
\nthalias{propositionN}
\nthalias{corollaryN}
\nthalias{lemma}
\nthalias{proposition}
\nthalias{corollary}
\nthalias{theorem}
\nthalias{conjecture}
\nthalias{question}
\nthalias{assumption}

\creflabelformat{enumi}{#2#1#3}

\crefname{section}{Section}{Sections}
\crefformat{section}{#2Section~#1#3} 
\Crefformat{section}{#2Section~#1#3} 

\crefname{subsection}{\S}{\S\S}
\AtBeginDocument{%
  \crefformat{subsection}{#2\S#1#3}%
  \Crefformat{subsection}{#2\S#1#3}%
}

\crefname{subsubsection}{\S}{\S\S}
\AtBeginDocument{%
  \crefformat{subsubsection}{#2\S#1#3}%
  \Crefformat{subsubsection}{#2\S#1#3}%
}

\theoremstyle{plain}

\newtheorem{lemma}{Lemma}[section]
\newtheorem{proposition}[lemma]{Proposition}
\newtheorem{corollary}[lemma]{Corollary}
\newtheorem{theorem}[lemma]{Theorem}

\theoremstyle{plain}
\theoremnumbering{Alph}

\theoremstyle{plain}
\theorembodyfont{\upshape}
\theoremsymbol{\ensuremath{\blacklozenge}}

\newtheorem{definition}[lemma]{Definition}

\newtheorem{remark}[lemma]{Remark}
\newtheorem{remarks}[lemma]{Remarks}

\crefname{definition}{definition}{definitions}
\crefformat{definition}{#2definition~#1#3} 
\Crefformat{definition}{#2Definition~#1#3} 

\crefname{ex}{example}{examples}
\crefformat{example}{#2example~#1#3} 
\Crefformat{example}{#2Example~#1#3} 

\crefname{exs}{example}{examples}
\crefformat{examples}{#2example~#1#3} 
\Crefformat{examples}{#2Example~#1#3} 

\crefname{remark}{remark}{remarks}
\crefformat{remark}{#2remark~#1#3} 
\Crefformat{remark}{#2Remark~#1#3} 

\crefname{remarks}{remark}{remarks}
\crefformat{remarks}{#2remark~#1#3} 
\Crefformat{remarks}{#2Remark~#1#3} 

\crefname{convention}{convention}{conventions}
\crefformat{convention}{#2convention~#1#3} 
\Crefformat{convention}{#2Convention~#1#3} 

\crefname{notation}{notation}{notations}
\crefformat{notation}{#2notation~#1#3} 
\Crefformat{notation}{#2Notation~#1#3} 

\crefname{table}{table}{tables}
\crefformat{table}{#2table~#1#3} 
\Crefformat{table}{#2Table~#1#3}

\crefname{lemma}{lemma}{lemmas}
\crefformat{lemma}{#2lemma~#1#3} 
\Crefformat{lemma}{#2Lemma~#1#3} 

\crefname{proposition}{proposition}{propositions}
\crefformat{proposition}{#2proposition~#1#3} 
\Crefformat{proposition}{#2Proposition~#1#3} 

\crefname{propositionN}{proposition}{propositions}
\crefformat{propositionN}{#2proposition~#1#3} 
\Crefformat{propositionN}{#2Proposition~#1#3} 

\crefname{corollary}{corollary}{corollaries}
\crefformat{corollary}{#2corollary~#1#3} 
\Crefformat{corollary}{#2Corollary~#1#3} 

\crefname{corollaryN}{corollary}{corollaries}
\crefformat{corollaryN}{#2corollary~#1#3} 
\Crefformat{corollaryN}{#2Corollary~#1#3} 

\crefname{theorem}{theorem}{theorems}
\crefformat{theorem}{#2theorem~#1#3} 
\Crefformat{theorem}{#2Theorem~#1#3} 

\crefname{theoremN}{theorem}{theorems}
\crefformat{theoremN}{#2theorem~#1#3} 
\Crefformat{theoremN}{#2Theorem~#1#3} 

\crefname{enumi}{}{}
\crefformat{enumi}{#2#1#3}
\Crefformat{enumi}{#2#1#3}

\crefname{assumption}{assumption}{Assumptions}
\crefformat{assumption}{#2assumption~#1#3} 
\Crefformat{assumption}{#2Assumption~#1#3} 

\crefname{construction}{construction}{Constructions}
\crefformat{construction}{#2construction~#1#3} 
\Crefformat{construction}{#2Construction~#1#3} 

\crefname{sketch}{sketch}{Sketches}
\crefformat{sketch}{#2sketch~#1#3} 
\Crefformat{sketch}{#2Sketch~#1#3} 

\crefname{question}{question}{Questions}
\crefformat{question}{#2question~#1#3} 
\Crefformat{question}{#2Question~#1#3} 

\crefname{equation}{}{}
\crefformat{equation}{(#2#1#3)} 
\Crefformat{equation}{(#2#1#3)}

\numberwithin{equation}{section}

\theoremstyle{nonumberplain}
\theoremsymbol{\ensuremath{\blacksquare}}

\newtheorem{proof}{Proof}
\newcommand\pf[1]{\newtheorem{#1}{Proof of \Cref{#1}}}

\newcommand\bZ{{\mathbb Z}}

\newcommand\cC{{\mathcal C}}

\newcommand\cM{{\mathcal M}}

\newcommand\cP{{\mathcal P}}

\newcommand\fg{{\mathfrak g}}

\newcommand\fn{{\mathfrak n}}

\newcommand\fr{{\mathfrak r}}

\newcommand\fz{{\mathfrak z}}

\newcommand\fsl{\mathfrak{sl}}

\DeclareMathOperator{\id}{id}

\DeclareMathOperator{\End}{\mathrm{End}}
\DeclareMathOperator{\Hom}{\mathrm{Hom}}

\newcommand{\cat}[1]{\textsc{#1}}

\newcommand{\qedhere}{\mbox{}\hfill\ensuremath{\blacksquare}}

\newcommand{\comment}[1]{}

\title{Associative/Lie radical (co)invariance under tracial (co)actions}
\author{Alexandru Chirvasitu}

\begin{document}

\date{}

\newcommand{\Addresses}{{
  \bigskip
  \footnotesize

  \textsc{Department of Mathematics, University at Buffalo}
  \par\nopagebreak
  \textsc{Buffalo, NY 14260-2900, USA}  
  \par\nopagebreak
  \textit{E-mail address}: \texttt{achirvas@buffalo.edu}

}}

\maketitle

\begin{abstract}
  We prove a number of Jacobson/solvable/nilpotent-radical (co)invariance results for finite-dimensional associative or Lie (co)module-algebras over Hopf algebras $H$, provided the characteristic, if positive, is large relative to the dimension of the algebra and the (co)actions satisfy trace-preservation conditions automatic when $H$ is involutory. This generalizes a number of such radical-invariance criteria in the literature, due to A. Gordienko, V. Linchenko, Pagon-Repov{\v{s}}-Zaicev and others, providing a uniform enriched-categorical framework for those results.
\end{abstract}


\vspace{.5cm}

\noindent{MSC 2020: 16T05; 16T15; 18D15; 17A65; 16N40; 18M05; 17B10; 15A15
  
  
}


\section*{Introduction}

Recall \cite[Theorem 2.1]{MR1995055}, to the effect that \emph{involutory} (i.e. square-antipode-1) Hopf algebras acting (unitally and multiplicatively) on finite-dimensional unital associative $\Bbbk$-algebras $A$ of \emph{characteristic exponent} \cite[\S I.3.1]{brl}
\begin{equation*}
  \exp(\Bbbk)
  :=
  \begin{cases}
    p&\mathrm{char}(\Bbbk)=p>0\\
    1&\mathrm{char}(\Bbbk)=0
  \end{cases}
\end{equation*}
coprime to $\dim A$ leave Jacobson radicals invariant. That claim notwithstanding, in positive characteristic the proof appears \cite{gord_msg_involut} to use the stronger requirement that $\dim A<\mathrm{char}(\Bbbk)$. In that form, the crux of the proof seems to consist in the $H$-action's being automatically \emph{tracial} when $H$ is involutory (\Cref{cor:h.inv.assoc}):

\begin{definition}\label{def:trcl.coact}
  Let $\tensor*[_H]{\cM}{_f}$, $\cM_f^H$ be the categories of finite-dimensional left $H$-modules and respectively right $H$-comodules for a Hopf $\Bbbk$-algebra over a field $\Bbbk$.

  \begin{enumerate}[(1),wide]
  \item An object $V$ in either category is \emph{tracial} if the induced (co)action on $\End(V)\cong V\otimes V^*$ preserves the trace functional $V\otimes V^*\xrightarrow{\mathrm{Tr}}\Bbbk$.

  \item An object in either of the categories $\tensor*[_H]{\cM}{}$ or $\cM^H$ of possibly infinite-dimensional (co)modules is \emph{(co-)hereditarily tracial} if all of its finite-dimensional subobjects (respectively quotients) are tracial.
  \end{enumerate}  
\end{definition}  
The motivating result thus generalizes:

\begin{theorem}\label{th:leaves.jac.inv}
  Let $A$ be a finite-dimensional unital, associative, tracial $H$-(co)module $\Bbbk$-algebra in the sense of \Cref{def:trcl.coact} with $\dim A<\mathrm{char}(\Bbbk)$ if the latter is positive.

  The Jacobson radical of $A$ is then $H$-(co)invariant. 
\end{theorem}

There is a Lie-algebraic version of the discussion, implying \emph{nilradical} \cite[\S I.7, post Proposition 6]{jac_lie_1979} invariance under appropriate conditions comparable to those assumed in \Cref{th:leaves.jac.inv}. 

\begin{theorem}\label{th:leaves.lie.nil.rad.inv}
  Let $L$ be a finite-dimensional, $H$-(co)module $\Bbbk$-Lie algebra, with co-hereditarily tracial tensor algebra $TL$ in the sense of \Cref{def:trcl.coact}, and with $\dim \mathrm{alg}\Braket{\mathrm{ad}_L(L)}<\mathrm{char}(\Bbbk)$ if the latter is positive.

  The nilradical $\fn(L)$ is then $H$-(co)invariant.
\end{theorem}

A variant addresses \emph{solvable-}radical preservation in characteristic 0.

\begin{theorem}\label{th:char0.solv.rad.trcl}
  For a finite-dimensional, co-hereditarily tracial $H$-(co)module characteristic-0 $\Bbbk$-Lie algebra $L$ the radical $\fr(L)$ is $H$-(co)invariant.   
\end{theorem}

A number of remarks note attendant pathologies:
\begin{itemize}[wide]
\item quantifiably ``large'' sets of prime characteristics house comparatively small non-solvable Lie algebras with vanishing Killing form (\Cref{pr:sl2.prime.dens}), impeding the recovery of the solvable radical via familiar Killing-kernel constructs;

\item and, while nilradical $H$-(co)invariance automatically entails radical (co)invariance in characteristic 0 (with no tracial assumptions), the converse is not generally true (\Cref{res:trc:nec}\Cref{item:res:trc:nec:solv.inv.nil.ninv}). 
\end{itemize}

\subsection*{Acknowledgments}

I am grateful for insightful comments from A. Gordienko. 


\section{Radical (co)invariance under tracial (co)actions}\label{se:rad.coinv}

\pf{th:leaves.jac.inv}
\begin{th:leaves.jac.inv}
  The hypothesis ensures that $H$ leaves (co)-invariant the bilinear form
  \begin{equation*}
    A^{\otimes 2}
    \ni
    a\otimes b
    \xmapsto{\quad}
    \Braket{a\mid b}
    :=
    \mathrm{Tr}(ab\cdot)
    \in
    \Bbbk,
  \end{equation*}
  and the conclusion follows from the fact that under the characteristic-index assumptions $J(A)$ is precisely the \emph{kernel} of that bilinear map:
  \begin{equation*}
    J(A)=
    \left\{x\in A\ :\ \Braket{x\mid a}=0,\ \forall a\in A\right\}
  \end{equation*}
  (a left-right symmetric condition, $\Braket{-\mid -}$ being symmetric). 
\end{th:leaves.jac.inv}

As noted prior to \Cref{def:trcl.coact}, the involutory-$H$ result (\cite[Theorem 2.1]{MR1995055}, \cite[Theorem 3.2]{MR3514537}) now follows.

\begin{corollary}\label{cor:h.inv.assoc}
  Let $A$ be a finite-dimensional unital, associative $H$-(co)module $\Bbbk$-algebra for an involutory Hopf algebra $H$ with $\dim A<\mathrm{char}(\Bbbk)$ if the latter is positive.

  The (co)action is automatically tracial, and hence the Jacobson radical of $A$ is $H$-(co)invariant. 
\end{corollary}
\begin{proof}
  The second claim follows from the first by \Cref{th:leaves.jac.inv}, so it suffices to address the former. For coactions this is recorded as \cite[Lemma 3.1]{MR3514537}; for completeness, a monoidal-categorical basis-free justification runs as follows.

  \begin{itemize}[wide]
    \item $H$ being involutory is equivalent, via \emph{Tannaka reconstruction} \cite[Theorem 2.4.2]{schau_tann}, to the canonical identification $V\cong V^{**}$ being an $H$-comodule morphism for all $V\in \cM^H_f$ (category of finite-dimensional $H$-comodules).

      Equivalently, \emph{left and right duals} \cite[Definition 1.2.3]{schau_tann} coincide in $\cM^H_f$ via the standard vector-space double-dual identification.

    \item now, the trace of an endomorphism $f\in \End(V)$, $\dim_{\Bbbk}<\infty$ can be defined as
      \begin{equation*}
        \Bbbk
        \xrightarrow[\quad\text{regarding $V^*$ as the \emph{left} dual}\quad]{\quad\text{coevaluate}\quad}
        V\otimes V^*
        \xrightarrow{\quad f\otimes\id\quad}
        V\otimes V^*
        \cong
        V^{**}\otimes V^*
        \xrightarrow{\quad\text{evaluate}\quad}
        \Bbbk.
      \end{equation*}
    \item All of the above being $H$-comodule maps in the involutory case, all $V\in \cM^H_f$ are tracial. 
    \end{itemize}
    Virtually the same argument functions for $H$-modules: in that case Tannaka reconstruction functions less smoothly, and the category recovered from the forgetful functor $\tensor*[_{H}]{\cM}{_f}\xrightarrow{\cat{forget}}\cat{Vec}$ from finite-dimensional $H$-modules is that of finite-dimensional comodules over the \emph{finite dual} \cite[Definition 1.2.3]{mont} $H^{\circ}$; nevertheless, $H$ being involutory certainly entails all of the categorical machinery adduced above for coactions.
\end{proof}

\pf{th:leaves.lie.nil.rad.inv}
\begin{th:leaves.lie.nil.rad.inv}
  The adjoint representation $L\xrightarrow{\mathrm{ad}_{L}(\bullet):=[\bullet,-]}\End(L)$ is $H$-equivariant, and the nilradical is precisely \cite[\S II.3, Theorem 3]{jac_lie_1979} the preimage of the Jacobson radical of the associative algebra generated by $\mathrm{ad}_L(L)$. As that associative algebra is a quotient of the tensor algebra $TL$ ($H$-equivariantly), the conclusion follows from the assumed co-hereditary tracial character of the (co)action and \Cref{th:leaves.jac.inv}. 
\end{th:leaves.lie.nil.rad.inv}

\begin{remark}\label{re:series.inv}
  As an aside on (co)invariant canonical constructs in Lie algebras, note that the respective members $L^{(n)}$, $L^{[n]}$ and $L_{[n]}$ of the \emph{derived, descending (lower) and ascending (upper) central series} \cite[\S\S I.1.5, I.1.6]{bourb_lie_1-3} always are for Lie algebras $L\in \tensor*[_H]{\cM}{}$ or $L\in \cM^H$ with $H$-(co)invariant bracket.

  For the derived series $L^{(n+1)}=\left[L^{(n)},L^{(n)}\right]$ and the lower central series $L^{[n+1]}=\left[L,L^{[n]}\right]$ this is immediate. For the upper central series, once the center $\fz(L)=L_{[1]}$ is proven (co)invariant the conclusion follows by iteration: the hypotheses assumed of $L$ will hold with $L/\fz(L)$ in its place. To prove center (co)invariance, observe that it is precisely the kernel of the $H$-(co)module morphism $L\xrightarrow{ad_L}\underline{\End}(L)$, the latter object denoting the \emph{internal hom} \cite[\S 1.5]{kly} in the appropriate \emph{closed monoidal} category:
  \begin{itemize}[wide]
  \item in $\tensor*[_H]{\cM}{}$, nothing but the full endomorphism ring $\End(L)$;
  \item in $\cM^H$, the largest \emph{rational} $H^*$-submodule (what \cite[Theorem 2.1.3(d)]{swe} would denote by $\End(L)^{rat}$).
  \end{itemize}
\end{remark}

In characteristic 0 more is true: \emph{(solvable-)radical} \cite[\S I.7, post Proposition 4]{jac_lie_1979} invariance is automatic under the theorem's constraints.

\begin{corollary}\label{cor:both.inv.char0}
  Under the hypotheses of \Cref{th:leaves.lie.nil.rad.inv}, in characteristic 0 the solvable radical $\fr(L)\trianglelefteq L$ is also $H$-(co)invariant.
\end{corollary}
\begin{proof}
  This follows from \Cref{th:leaves.lie.nil.rad.inv} and \Cref{cor:if.nil.then.solv} below, recording a more general invariance phenomenon.
\end{proof}

\begin{lemma}\label{le:mod.quot.inv}
  Let $L\in \cC\in \left\{\tensor*[_H]{\cM}{},\cM^H\right\}$ be a Lie algebra with invariant bracket for a Hopf algebra $H$. If $M,N\le P$ are $H$-equivariant $L$-module embeddings, then
  \begin{equation*}
    (M:N)
    :=
    \left\{x\in L\ :\ xN\le M\right\}
    \overset{\text{ideal}}{\trianglelefteq}
    L
  \end{equation*}
  is $H$-(co)invariant. 
\end{lemma}
\begin{proof}
  Simply observe that $(M:N)$ can be recovered as the pullback
  \begin{equation*}
    \begin{tikzpicture}[>=stealth,auto,baseline=(current  bounding  box.center)]
      \path[anchor=base] 
      (0,0) node (l) {$(M:N)$}
      +(3,.5) node (u) {$L$}
      +(3,-.5) node (d) {$\underline{\Hom}(N,M)$}
      +(6,0) node (r) {$\underline{\Hom}(N,P)$,}
      ;
      \draw[right hook->] (l) to[bend left=6] node[pos=.5,auto] {$\scriptstyle \trianglelefteq$} (u);
      \draw[->] (u) to[bend left=6] node[pos=.5,auto] {$\scriptstyle $} (r.north west);
      \draw[->] (l) to[bend right=6] node[pos=.5,auto,swap] {$\scriptstyle $} (d);
      \draw[right hook->] (d) to[bend right=6] node[pos=.5,auto,swap] {$\scriptstyle $} (r.south west);
    \end{tikzpicture}
  \end{equation*}
  $\underline{Hom}(\bullet,\bullet)\in \cC$ again denoting the internal homs referenced in \Cref{re:series.inv}. The conclusion follows from the fact that the right-hand maps are, in this case, $H$-equivariant (and the forgetful functor $\cC\to \cat{Vec}_{\Bbbk}$ is exact and hence pullback-preserving). 
\end{proof}

\begin{corollary}\label{cor:if.nil.then.solv}
  Let $H$ be a characteristic-0 Hopf algebra and $L\in \cM^H_f$ or $L\in \tensor*[_H]{\cM}{_f}$ a Lie algebra with $H$-(co)invariant bracket.

  If the nilradical of $L$ is $H$-(co)invariant, so is the solvable radical. 
\end{corollary}
\begin{proof}
  In characteristic 0 $\fr(L)\le (\fn(L):L)$ \cite[\S II.7, Theorem 13]{jac_lie_1979} on the one hand, while the opposite inclusion is automatic regardless of characteristic; the claim follows from \Cref{le:mod.quot.inv}. 
\end{proof}

\pf{th:char0.solv.rad.trcl}
\begin{th:char0.solv.rad.trcl}
  Observe first that co-hereditarily tracial $H$-(co)module Lie algebras have $H$-(co)invariant \emph{Killing form} \cite[\S 5.1]{hmph_1972}: the composition
  \begin{equation*}
    L^{\otimes 2}
    \xrightarrow{\quad\mathrm{ad}_L^{\otimes 2}\quad}
    \left(L\otimes L^*\right)^{\otimes 2}
    \xrightarrow{\quad\text{multiply}\quad}
    L\otimes L^*
    \xrightarrow{\quad\mathrm{Tr}\quad}
    \Bbbk.
  \end{equation*}
  The kernel
  \begin{equation*}
    \ker \kappa:=\left\{x\in L\ :\ \kappa(x,-)\equiv 0\right\}
    \trianglelefteq
    L
  \end{equation*}
  of $\kappa$ on the one hand plainly contains the nilradical $\fn(L)$, and is on the other hand contained in the radical $\fr(L)$ by \cite[\S 5.1, proof of Theorem]{hmph_1972} (which requires the characteristic 0 assumption). The induced (co)action on $L/\ker \kappa$ is again tracial by the assumed hereditarily tracial character of the original (co)action, so we can proceed by induction on $\dim L$ by passing to $L/\ker\kappa$ for as long as the Killing form $\kappa$ is degenerate.
  
  As soon as the induced $\kappa$ on one of the successive quotients of $L$ is non-degenerate, that quotient will be semisimple by \cite[\S 5.1, Theorem]{hmph_1972} (Killing non-degeneracy implies semisimplicity regardless of characteristic, as that result's proof notes). The $H$-(co)invariant union of the preimage chain of induced Killing kernels will be the radical $\fr(L)$.
\end{th:char0.solv.rad.trcl}

In parallel to \Cref{cor:h.inv.assoc}, there is the analogous consequence with essentially the same proof.

\begin{corollary}\label{cor:h.inv.lie}
  Let $L$ be a finite-dimensional $H$-(co)module Lie $\Bbbk$-algebra for an involutory Hopf algebra $H$ with $\dim \mathrm{alg}\Braket{\mathrm{ad}_L(L)}<\mathrm{char}(\Bbbk)$ if the latter is positive.

  \begin{enumerate}[(1),wide]
  \item The (co)action is automatically co-hereditarily tracial on $TL$. Consequently:

  \item The nilradical $\fn(L)$ is $H$-(co)invariant.

  \item In characteristic 0 the radical $\fr(L)$ is $H$-(co)invariant.  \qedhere
  \end{enumerate}  
\end{corollary}

\begin{remark}\label{re:lie.rad.inv.char0}
  Literature precursors to \Cref{cor:h.inv.lie} include (nil)radical coinvariance under involutory coactions in characteristic 0 \cite[Theorem 4.1]{MR3514537}, in turn generalizing automatic (nil)radical grading (\cite[Corollary 4.3]{MR3115171}, \cite[Proposition 3.3]{MR3116786}) in algebraically closed characteristic 0.
\end{remark}

Characteristic 0 is essential in concluding that the iterative procedure of extending Killing-form kernels (employed in the proof of \Cref{th:char0.solv.rad.trcl}) functions, even for non-solvable Lie algebras small relative to the positive characteristic. \Cref{pr:sl2.prime.dens} is one incarnation of those types of positive-characteristic pathologies; its statement references the usual \cite[\S VI.4.1]{ser_arith_1973} notion of \emph{density} for a set $\cP'\subseteq \cP:=\text{set of primes}$
\begin{equation*}
  \forall\left(\cP'\subseteq \cP:=\text{set of all primes}\right)
  \left(
    \text{\emph{density}}\quad d(\cP')
    \ 
    :\xlongequal{\text{if it exists}}
    \ 
    \lim_{s\searrow 1}\frac{\sum_{p\in \cP'}\frac 1{p^s}}{\log \frac 1{s-1}}
  \right).
\end{equation*}

\begin{proposition}\label{pr:sl2.prime.dens}
  The set of primes
  \begin{equation*}
    \cP':=
    \left\{
      p\in \cP\ :\
      \begin{gathered}
        \exists\left(\text{representation }\rho:\fsl_2\circlearrowright V\right)/\text{algebraically closed $\mathrm{char}(\Bbbk)=p$}\\
        \kappa_L\equiv 0
        ,\quad
        \dim L < p
        ,\quad
        L:=V\rtimes \fsl_2
      \end{gathered}
    \right\}
  \end{equation*}
  contains some
  \begin{equation*}
    \cP''\subseteq \cP'
    ,\quad
    d\left(\cP''\right)
    \ge
    \frac 1{\inf_{n\in 2\bZ_{\ge 1}+1}\varphi\left(\frac{n(2n-1)(2n+1)}3\right)}
    ,\quad
    \varphi:=\text{Euler totient}.
  \end{equation*}
\end{proposition}
\begin{proof}
  Note first that for any \emph{semidirect product} (\cite[\S I.1.8, Example 2]{bourb_lie_1-3}) $L:=V\rtimes \fg$ attached to a representation $\rho:\fg\circlearrowright V$ with semisimple $\fg$ we have $\fr(L)=\fn(L)=V$, so the vanishing of the Killing form $\kappa_L$ amounts to its vanishing on ${\fg}^{\otimes 2}$. There, it restricts to
  \begin{equation}\label{eq:res.kill}
    \kappa_L|_{\fg^{\otimes 2}}
    =
    \kappa_{\fg}
    +\kappa_{\rho}
    ,\quad
    \kappa_{\rho}
    :=
    \mathrm{Tr}\left(\rho(-)\cdot \rho(-)\right)
    \quad
    \left(\text{\emph{trace form} \cite[p.543]{MR2578888}}\right).
  \end{equation}
  In the present context, the various $\rho:\fsl_2\circlearrowright V$ to be considered will be multiples $V_m^{\oplus k}$ of the usual \cite[\S 7.2, Theorem]{hmph_1972} Weyl modules $V_m$, $\dim V_m=m+1$, exhausting the simple finite-dimensional $\fsl_2$-representations in characteristic 0.
  
  It follows from \cite[Proposition 4.1]{MR2578888} and \Cref{eq:res.kill} that $\kappa_{V^{\oplus k}_{2n-1}\rtimes \fsl_2}\equiv 0$ in characteristic $p$ whenever 
  \begin{equation*}
    p\ |\ 4+k\left(1^2+3^2+\cdots+(2n-1)^2\right)
    =
    4+k\frac{n(2n-1)(2n+1)}3.
  \end{equation*}
  Per \emph{Dirichlet's theorem} \cite[\S VI.4.1, Theorem 2]{ser_arith_1973} on primes in arithmetic progressions, that last value achieves prime values ranging over a set $\cP''$ of density $\frac 1{\varphi\left(\frac{n(2n-1)(2n+1)}3\right)}$. The dimension constraint
\begin{equation*}
  \dim L=\dim \fsl_2+k\dim V_{2n-1}=3+2kn
  <
  4+k\frac{n(2n-1)(2n+1)}3
\end{equation*}
also obtains for the stipulated range of $n\ge 3$, hence the conclusion.
\end{proof}

\begin{remarks}\label{res:trc:nec}
  \begin{enumerate}[(1),wide]
  \item Field-characteristic considerations aside, the tracial character of the (co)actions involved is also a crucial factor in radical (co)invariance. \cite[Example 2.8]{MR3115171} can be expanded (and paraphrased) into an action by \cite[\S 7.3]{rad}'s 4-dimensional Hopf algebra $H_{2,-1}$ on $L:=\fg\rtimes \fg$ (semidirect product attached to the adjoint action) for any semisimple $\fg$, with
    \begin{itemize}[wide]
    \item an order-2 \emph{grouplike} \cite[Deﬁnition 2.1.10]{rad} $g\in H_{2,-1}$ implementing a $\bZ/2$-grading rendering the acting (non-ideal) $\fg$ even and the acted-upon $\fg$ odd;
    \item and a nilpotent $x\in J(H_{2,-1})$ acting as an odd operator, annihilating the even $\fg$ and mapping the odd $\fg$($=\fr(L)=\fn(L)$) onto the former via the obvious identification. 
    \end{itemize}
    The cited \cite[Example 2.8]{MR3115171} is what this construction specializes to (in characteristic 0, and presented in slightly different language) for $\fg:=\fsl_2$.

  \item That example generalizes further, with $\fg$ still assumed semisimple: for \emph{any} representation $\fg\circlearrowright V$ and $\fg$-equivariant map $V\xrightarrow{\psi}\fg$ (the codomain being $\fg$-acted upon adjointly), compatible in the sense that
    \begin{equation*}
      \forall\left(v,w\in V\right)
      \left(
        \psi(v)\triangleright w+\psi(w)\triangleright v=0
      \right)
    \end{equation*}
    (a condition reminiscent of the \emph{Peiffer identity} \cite[p.462]{MR2068521} familiar in the theory of 2-groups and crossed modules), one has an action $H_{2,-1}\circlearrowright L:=V\rtimes \fg$ with $g$ implementing a grading as before and $x$ operating as $\psi$ on $V=\fr(L)=\fn(L)\le L$ and trivially on $\fg\le L$. $\psi\ne 0$ violates (plain and nil)radical preservation.
    
  \item\label{item:res:trc:nec:solv.inv.nil.ninv} Effectively the same argument, incidentally, also shows that the converse to \Cref{cor:if.nil.then.solv} cannot hold (even in characteristic zero, barring involutivity): in the same construction, take $\fg$ solvable (so that $L:=V\rtimes \fg$ is as well) but not nilpotent. If the image of the $\fg$-equivariant map $V\xrightarrow{\psi}\fg$ is not contained in $\fn(\fg)$, the nilradical $\fn(L)\le V\oplus \fn(\fg)$ will not be $H_{2,-1}$-invariant despite $\fr(L)=L$ being so. The requisite constraints are achievable: $\fg$ might be solvable non-nilpotent, $\fg\circlearrowright V:=\fg$ the adjoint representation, and $\psi=\id_{\fg}$. 
  \end{enumerate}
\end{remarks}


\addcontentsline{toc}{section}{References}

\def\polhk#1{\setbox0=\hbox{#1}{\ooalign{\hidewidth
  \lower1.5ex\hbox{`}\hidewidth\crcr\unhbox0}}}


\Addresses

\end{document}